\documentclass{amsart}
\usepackage{amsfonts}
\usepackage{amssymb}
\usepackage{graphicx}
\usepackage{amscd}

\newtheorem{theorem}{Theorem}
\theoremstyle{plain}

\newtheorem{definition}{Definition}

\numberwithin{equation}{section}

\input{tcilatex}

\begin{document}
\title[On an ntegral as an interval function]{On an integral as an interval
function}
\author{Branko Sari\'{c}}
\address{Faculty of Sciences, University of Novi Sad. Trg Dositeja
Obradovica 2, 21000 Novi Sad, Serbia; College of Technical Engineering
Professional Studies, Svetog Save 65, 32 000 \v{C}a\v{c}ak, Serbia}
\email{saric.b@open.telekom.rs}
\date{September 13, 2014}
\thanks{The author's research is supported by the Ministry of Science,
Technology and Development, Republic of Serbia (Project ON 174024)}

\begin{abstract}
Based on the\ total integrability we first define an integral of a real
valued function $f$ as an interval function associated to its antiderivative 
$F$. By introducing the concept of the residue of a function into the real
analysis, the relationship between the so defined integral and the
generalized\textit{\ Riemann} integral is established.
\end{abstract}

\maketitle

\section{Introduction}

An antiderivative of a real-valued function $f$ is just a function $F$ whose
derivative is $f$. The collection of functions $F+C$, where $C$ is an
arbitrary constant known as the constant of integration, is a nonunique
inverse of the derivative $f$. Another way of stating this is that the set
of all antiderivatives $F+C$ is an indefinite integral of $f$. In symbols, $%
\int f\left( x\right) dx=F+C$. So, the opposite process to differentiation
is integration. The fundamental theorem of calculus, more precisely its the
second part, allows definite integrals to be computed in terms of indefinite
integrals. This part of the theorem states that if $F$ is the antiderivative
for $f$, then, under certain conditions, the definite integral of $f$ over a
compact interval $I\subset \mathbb{R}$ is equal to the difference between
the values of an antiderivative $F$ evaluated at the endpoints of the
interval. In symbols, $\int_{I}f\left( x\right) dx=\Delta F\left( I\right) $%
. Here, $\Delta F$ is an associated interval function of $F$, such that $%
\Delta F\left( I\right) =F\left( v\right) -F\left( u\right) $ for any
compact interval $I=\left[ u,v\right] $, \cite{6}. Obviously, if $F$ is
defined on $I$, then the sum of the changes in the value of $F$ over $I$
with any partition is equal to $\Delta F\left( I\right) $. Hence, an attempt
has been made by \textit{Sari\'{c}} \cite{4,5} to define an integral of $f$
over $I$, as the sum of these changes in the value of $F$ over $I$, for
which the \textit{Newton--Leibniz} formula (the second part of the
fundamental theorem of calculus) to be valid unconditionally. The resulting
integral is the so-called total \textit{Kurzweil}-\textit{Henstock}
integral. Accordingly, instead of the set of functions $F+C$ we can use the
associated interval function $\Delta F$ of $F$ to be an integral of $f$. In
symbols, $\int f\left( x\right) dx:=\Delta F$. Therefore, the purpose of
this note is to convert the fundamental theorem of calculus into the
definition of integrability of $f$, as follows.

\begin{definition}
Let $f$ be a real valued function with antiderivative $F$ and let $\mathcal{I%
}$ be any collection of compact intervals $I$ of the real line $\mathbb{R}$.
If $\Delta F:\mathcal{I}\rightarrow \mathbb{R}$ is the associated interval
function of $F$, then $f$ is integrable to $\Delta F\left( I\right) $ on $%
I\in \mathcal{I}$. In symbols, $\int_{I}f\left( x\right) dx=\Delta F\left(
I\right) $.
\end{definition}

When working with functions, which have a finite number of discontinuities
on the compact interval $\left[ a,b\right] \subset \mathbb{R}$, it does not
really matter how these functions will be defined on the set $E$ of
discontinuities. Unless otherwise stated in what follows, we assume that the
endpoints of $[a,b]$ do not belong to $E$. As this situation will arise
frequently, we adopt the convention that such functions are equal to $0$ at
all points at which they have an infinite value ($\pm \infty $) or not be
defined at all. Hence, we may define point functions $F_{ex}:\left[ a,b%
\right] \mapsto \mathbb{R}$ and $f_{ex}:\left[ a,b\right] \mapsto \mathbb{R}$
by extending $F$ and its derivative $f$ from $\left[ a,b\right] \backslash E$
to $E$ by $F_{ex}\left( x\right) =0$ and $f_{ex}\left( x\right) =0$ for $%
x\in E$, so that 
\begin{equation}
F_{ex}\left( x\right) =\left\{ 
\begin{array}{c}
F\left( x\right) \text{, if }x\in \left[ a,b\right] \backslash E \\ 
0\text{, if }x\in E%
\end{array}%
\right. \text{ and}  \label{1}
\end{equation}%
\begin{equation*}
f_{ex}\left( x\right) =\left\{ 
\begin{array}{c}
f\left( x\right) \text{, if }x\in \left[ a,b\right] \backslash E \\ 
0\text{, if }x\in E%
\end{array}%
\right. \text{.}
\end{equation*}

If we denote any generalized \textit{Riemann} integral of~$f_{ex}$ on $I$ by 
$\mathcal{R}-\int_{I}f_{ex}\left( x\right) dx$, including the \textit{Riemann%
} integral itself, then we will prove below the following result $\Delta
F\left( I\right) =\int_{I}f\left( x\right) dx=\mathcal{R}-\int_{I}f_{ex}%
\left( x\right) dx+\Re $, where $\Re $ is the sum of residues of $F$ on the
set $E\subset I$ at whose points $F$ is not differentiable. Clearly, if $%
\mathcal{R}-\int_{I}f_{ex}\left( x\right) dx$ does not exist, then $\mathcal{%
R}-\int_{I}f_{ex}\left( x\right) dx+\Re $ is reduced to the so-called
indeterminate expression $\infty -\infty $ that actually have, in this
situation, the real numerical value of $\Delta F(I)$.

\section{Preliminaries}

Given a compact interval $\left[ a,b\right] $ in $\mathbb{R}$, let the
collection $\mathcal{I}\left( \left[ a,b\right] \right) $ be a family of all
compact subintervals $I$ of $\left[ a,b\right] $. The \textit{Lebesgue}
measure in \thinspace $\mathbb{R}$ is denoted by $\mu $, however, for $%
I\subset \mathbb{R}$ we write $\left\vert I\right\vert =\Delta x\left(
I\right) $ instead of $\mu \left( I\right) $.

A partition $P[a,b]$ of a compact interval $[a,b]\in \mathbb{R}$ is a finite
set (collection) of interval-point pairs $([a_{i},b_{i}],x_{i})_{i\leq \nu }$%
, such that the subintervals $[a_{i},b_{i}]$ are non-overlapping, $\cup
_{i\leq \nu }[a_{i},b_{i}]=[a,b]$ and $x_{i}\in \lbrack a,b]$. The points $%
\{x_{i}\}_{i\leq \nu }$ are the tags of $P[a,b]$, \cite{1}. It is evident
that there are many different ways to arrange the position of the tags $%
x_{i} $ with respect to $[a_{i},b_{i}]$. Each of these positions leads to
one of a \textit{Riemann }type definition of the generalized \textit{Riemann}
integral. If $E$ is a subset of $[a,b]$, then the restriction of $P[a,b]$ to 
$E$ is a finite collection of $([a_{i},b_{i}],x_{i})$ $\in P[a,b]$, such
that each pair of sets $[a_{i},b_{i}]$ and $E$ intersects in at least one
point and all $x_{i}$ are tagged in $E$. In symbols, $P[a,b]|_{E}=%
\{([a_{i},b_{i}],x_{i})\in P[a,b]\mid \lbrack a_{i},b_{i}]\cap E\neq
\emptyset $ and $x_{i}\in E\}$. Given $\delta :[a,b]\mapsto \mathbb{R}_{+}$,
named a gauge, a point-interval pair $([a_{i},b_{i}],x_{i})$ is called $%
\delta $\textit{-fine} if $[a_{i},b_{i}]\subseteq (x_{i}-\delta
(x_{i}),x_{i}+\delta (x_{i}))$. Let $\mathcal{P}[a,b]$ be the family of all
partitions $P[a,b]$ of $[a,b]$. If $E\subseteq \lbrack a,b]$, then for any
position of the tags $x_{i}$ with respect to $[a_{i},b_{i}]$ the family of
all $\delta $\textit{-fine} partitions $P[a,b]$ of $\left[ a,b\right] $,
such that $P[a,b]|_{E}\subset P[a,b]$, denoted by $\mathcal{P}_{\delta }%
\left[ a,b\right] |_{E}$. In what follows we will use the following
notations: $\sum_{i}\Delta F_{ex}(\left[ a_{i},b_{i}\right] )=\Delta F(P%
\left[ a,b\right] |_{E})$ and $\sum_{i}f_{ex}\left( x_{i}\right) \left\vert %
\left[ a_{i},b_{i}\right] \right\vert =\delta F(P\left[ a,b\right] |_{E})$,
whenever $(\left[ a_{i},b_{i}\right] ,x_{i})\in P\left[ a,b\right] |_{E}$.

\begin{definition}
Let $\varphi :\mathcal{I}\left[ a,b\right] \mapsto \mathbb{R}$ be an
arbitrary interval function and $E\subseteq \left[ a,b\right] $. A point
function $f:\left[ a,b\right] \mapsto \mathbb{R}$ is the limit of $\phi $ on 
$\left[ a,b\right] \backslash E$ if for every $\varepsilon >0$ there exists
a gauge $\delta :[a,b]\mapsto \mathbb{R}_{+}$, such that 
\begin{equation}
\left\vert \varphi ([a_{i},b])-f\left( x_{i}\right) \right\vert <\varepsilon 
\text{,}  \label{2}
\end{equation}%
whenever $([a_{i},b_{i}],x_{i})\in P\left[ a,b\right] \backslash P\left[ a,b%
\right] |_{E}$ and $P\left[ a,b\right] \in \mathcal{P}_{\delta }\left[ a,b%
\right] |_{E}$.
\end{definition}

Given a derivative-antiderivative pair ($f$ and $F$), the derivative $f$ is
the limit of the interval function 
\begin{equation}
\varphi \left( I\right) =\frac{\Delta F\left( I\right) }{\Delta x\left(
I\right) }=\frac{1}{\Delta x\left( I\right) }\int_{I}f\left( x\right) dx%
\text{,}  \label{3}
\end{equation}%
where $\Delta F\left( I\right) $ is the associated interval function of $F$.

\section{Main results}

Let $F:[a,b]\mapsto \mathbb{R}$. It is an old result that $F$ is continuous
on $[a,b]$ if and only if the associated interval function $\Delta F$ of $F$
converges to $0$ at all points of $\left[ a,b\right] $, \cite{3}.
Accordingly, we are now in a position to define the linear differential form
on $\left[ a,b\right] $.

\begin{definition}
For $F:\left[ a,b\right] \mapsto \mathbb{R}$ let $\varphi \ $be defined by $%
( $\ref{3}$)$. Then, $dF=fdx$ as the limit of the interval function 
\begin{equation}
\Delta F(I)=\int_{I}f\left( x\right) dx=\varphi (I)\Delta x(I)  \label{4}
\end{equation}%
on $\left[ a,b\right] $ is a linear differential form on $\left[ a,b\right] $%
.
\end{definition}

Clearly, if $F$ is continuous on $[a,b]$ then $dF=fdx$ vanishes identically
on $[a,b]$. In case $F$ is differentiable to $f$ everywhere on $[a,b]$
except for a set $E\subset \lbrack a,b]$ of \textit{Lebesgue} measure zero,
we can introduce into the analysis an interval-point function $\delta
F:[a,b]\times \mathcal{I}\left( \left[ a,b\right] \right) \mapsto \mathbb{R}$
being the product of the point function $f_{ex}$ defined by (\ref{1}) and
the interval function $\Delta x$, as follows%
\begin{equation}
\delta F(I,x)=f_{ex}\left( x\right) \Delta x(I)\text{.}  \label{5}
\end{equation}

As we can see, there is a difference between the interval-point function $%
\delta F(I,x)$ and the interval function $\Delta F(I)$, as well as between
their limits on $E$. However, by \textit{Definition 2}, since $f$ is the
limit of $\varphi $ on $\left[ a,b\right] \backslash E$, given $\varepsilon
>0$ there exists a gauge $\delta :[a,b]\mapsto \mathbb{R}_{+}$, such that 
\begin{equation}
\left\vert \delta F([a_{i},b_{i}],x_{i})-\Delta F([a_{i},b_{i}])\right\vert
<\varepsilon \Delta x([a_{i},b_{i}])\text{,}  \label{6}
\end{equation}%
whenever $([a_{i},b_{i}],x_{i})\in P\left[ a,b\right] \backslash P\left[ a,b%
\right] |_{E}$ and $P\left[ a,b\right] \in \mathcal{P}_{\delta }\left[ a,b%
\right] |_{E}$. So, in this emphasized case $fdx$ is the limit of both $%
\delta F$ and $\Delta F$ on $[a,b]\backslash E$.

Remember, there are many different ways to arrange the position of the tags $%
x_{i}$ with respect to $[a_{i},b_{i}]$, each of which leads to one type of
the generalized \textit{Riemann} integral defined by the following
definition.

\begin{definition}
For $\left[ a,b\right] \in \mathcal{I}\left[ a,b\right] $ let $E\subset %
\left[ a,b\right] $ be \textit{a set of Lebesgue measure zero} at whose
points a real valued function $f$ is not defined. A point function $%
f_{ex}:[a,b]\mapsto R$ is generalized Riemann integrable to a real point $%
\mathcal{F}$ on $\left[ a,b\right] $ if for every $\varepsilon >0$ there
exists a gauge $\delta :[a,b]\mapsto \mathbb{R}_{+}$, such that%
\begin{equation}
\left\vert \delta F(P\left[ a,b\right] )-\mathcal{F}\right\vert <\varepsilon 
\text{,}  \label{7}
\end{equation}%
whenever$P\left[ a,b\right] \in \mathcal{P}_{\delta }\left[ a,b\right] |_{E}$%
. In symbols, $\mathcal{F}:=$ $\mathcal{R}-\int_{a}^{b}f\left( x\right) dx$.
\end{definition}

If $x_{i}\in \lbrack a_{i},b_{i}]$ and the gauge $\delta (x)$ has a positive
infimum on $[a,b]$, then the previous definition\ becomes that of the
ordinary \textit{Riemann} integral.

The following two definitions introduce the concept of the residue of a
function into the real analysis.

\begin{definition}
For $\left[ a,b\right] \in \mathcal{I}\left[ a,b\right] $ let $E\subset %
\left[ a,b\right] $ be \textit{a set of Lebesgue measure zero} at whose
points a real valued function $F$ is not defined. The function $F$ is said
to be basically summable $($BS$_{\delta })$ on $E$ to a real number $\Re $,
if for every $\varepsilon >0$ there exists a gauge $\delta :[a,b]\mapsto 
\mathbb{R}_{+}$, such that $\left\vert \Delta F(P\left[ a,b\right]
|_{E})-\Re \right\vert <\varepsilon $, whenever $P\left[ a,b\right] \in 
\mathcal{P}_{\delta }\left[ a,b\right] |_{E}$. If in addition $E$ can be
written as a countable union of sets on each of which $F$ is BS$_{\delta }$,
then $F$ is said to be BSG$_{\delta }$ on $E$. In symbols, $\Re :=\sum_{x\in
E}f\left( x\right) dx$.
\end{definition}

If $F$ is absolutely continuous on $\left[ a,b\right] $, that means it has
negligible variation on $E$, then $\Re $ is equal to zero, \cite{1}.

\begin{definition}
The linear differential form $dF=fdx$ is a residue function of $F$. In
symbols, $\mathfrak{R}:=dF$.
\end{definition}

Obviously, the residue function of $F$ being basically summable $($BS$%
_{\delta })$ on $E\subset \left[ a,b\right] $ to a real number $\Re $ is a
null function on $\left[ a,b\right] $ (\textit{A function }$F:[a,b]\mapsto 
\mathbb{R}$\textit{\ is said to be a null function on }$[a,b]$,\textit{\ if
the set }$\{x\in \lbrack a,b]\mid F\left( x\right) \neq 0\}$\textit{\ is a
set of Lebesgue measure zero}, see \textit{2.4 Definition} in \cite{1}) and%
\begin{equation}
\Re =\sum_{x\in E}\mathfrak{R}\left( x\right) \text{.}  \label{12}
\end{equation}%
On the other hand, for some compact interval $\left[ a,b\right] \in \mathcal{%
I}\left[ a,b\right] $ the infinite sum $\sum_{x\in \left[ a,b\right] }%
\mathfrak{R}\left( x\right) $ is in fact the integral of $f$ on $\left[ a,b%
\right] $ since the antiderivative $F$ of $f$ is by \textit{Definition 5 }%
basically summable $($BS$_{\delta })$ on $\left[ a,b\right] $ to $\Delta
F\left( \left[ a,b\right] \right) $, so that%
\begin{equation}
\Delta F\left( \left[ a,b\right] \right) =\sum_{x\in \left[ a,b\right] }%
\mathfrak{R}\left( x\right) \text{.}  \label{13}
\end{equation}%
In case when $F$ has a certain number of discontinuities within $\left[ a,b%
\right] $, gathered together into the set $E\subset \left[ a,b\right] $, at
which its derivative $f$ can take values $\pm \infty $ or not be defined at
all, the sum $\sum_{x\in \left[ a,b\right] \backslash E}\mathfrak{R}\left(
x\right) $ reduces to the sum $\sum_{x\in \left[ a,b\right] }f_{ex}\left(
x\right) dx=\mathcal{R}-\int_{a}^{b}f_{ex}\left( x\right) dx$, since $%
f_{ex}dx$ is the limit of $\delta F$ on $\left[ a,b\right] $. Hence, if we
split the sum $\sum_{x\in \left[ a,b\right] }\mathfrak{R}\left( x\right) $
into two sums of $\mathfrak{R}\left( x\right) $ over two separate intervals $%
\left[ a,b\right] \backslash E$ and $E$, then we finally obtain that 
\begin{equation}
\int_{a}^{b}f\left( x\right) dx=\mathcal{R}-\int_{a}^{b}f_{ex}\left(
x\right) dx+\Re \text{.}  \label{14}
\end{equation}%
In what follows we shall formulate the result (\ref{14}) as a theorem and
prove it explicitly.

\begin{theorem}
For a some compact interval\ $\left[ a,b\right] \in \mathbb{R}$ let $%
E\subset \left[ a,b\right] $ be a set of \textit{Lebesgue} measure zero at
whose points a real valued function $F$ defined and differentiable on $\left[
a,b\right] \backslash E$ and its derivative $f$ can take values $\pm \infty $
or not be defined at all. If $F$ is basically summable $($BS$_{\delta })$ on 
$E$ to $\mathbf{\Re }$, then $f_{ex}$ is generalized Riemann integrable on $%
\left[ a,b\right] $ and%
\begin{equation}
\int_{a}^{b}f\left( x\right) dx=\mathcal{R}-\int_{a}^{b}f_{ex}\left(
x\right) dx+\Re \text{.}  \label{15}
\end{equation}
\end{theorem}

\begin{proof}
Let~$F_{ex}$ and $f_{ex}$ be defined by (\ref{1}). Since $F$ is BS$_{\delta }
$ on $E$ to $\mathbf{\Re }$ it follows from \textit{Definition 5 }that for
every $\varepsilon >0$ there exists a gauge $\delta :[a,b]\mapsto \mathbb{R}%
_{+}$, such that $\left\vert \Delta F(P\left[ a,b\right] |_{E})-\Re
\right\vert <\varepsilon $, whenever $P\left[ a,b\right] \in \mathcal{P}%
_{\delta }\left[ a,b\right] |_{E}$. In addition, $f_{ex}\left( x\right)
\equiv 0$ on $E$ and $\Delta F(P\left[ a,b\right] )=\Delta F\left( \left[ a,b%
\right] \right) $ whenever $P\left[ a,b\right] \in \mathcal{P}\left[ a,b%
\right] $. Hence, by the result (\ref{6}), for every $\varepsilon >0$ there
exists a gauge $\delta :[a,b]\mapsto \mathbb{R}_{+}$, such that 
\begin{equation*}
\left\vert \delta F(P\left[ a,b\right] )-[\Delta F\left( \left[ a,b\right]
\right) -\Re ]\right\vert \leq 
\end{equation*}%
\begin{equation*}
\leq \left\vert \delta F(P\left[ a,b\right] \backslash P\left[ a,b\right]
|_{E})-\Delta F(P\left[ a,b\right] \backslash P\left[ a,b\right]
|_{E})\right\vert +
\end{equation*}%
\begin{equation*}
+\left\vert \Delta F(P\left[ a,b\right] |_{E})-\Re \right\vert <\varepsilon
\left( \left\vert \left[ a,b\right] \right\vert +1\right) \text{,}
\end{equation*}%
whenever $P\left[ a,b\right] \in \mathcal{P}_{\delta }\left[ a,b\right] |_{E}
$. So, by \textit{Definition 4} $f_{ex}$ is generalized \textit{Riemann}
integrable on $\left[ a,b\right] $ and $\mathcal{R}-\int_{a}^{b}f_{ex}\left(
x\right) dx=\Delta F\left( \left[ a,b\right] \right) -\Re $, that is, 
\begin{equation*}
\int_{a}^{b}f\left( x\right) dx=\mathcal{R}-\int_{a}^{b}f_{ex}\left(
x\right) dx+\mathbf{\Re }\text{.}
\end{equation*}
\end{proof}

This result provides an extension of \textit{Cauchy's} result from the
calculus of residues in $\mathbb{R}$ (compare with results in \cite{4}).

\section{Example}

Let $C:\left[ 0,1\right] \mapsto $ $\mathbb{R}$ be the\textit{\ Cantor}
function, \cite{2}. Its derivative $c$ is a null function on $\left[ 0,1%
\right] $ that is not defined on the \textit{Cantor} set $\mathcal{C}$.
Since the generalized \textit{Riemann} integral of $c_{ex}$ $:\left[ 0,1%
\right] \mapsto $ $0$ on $\left[ 0,1\right] $ is equal to zero it follows
from the \textit{Theorem 1} that%
\begin{equation*}
\mathbf{\Re =}\int_{0}^{1}c\left( x\right) dx-\mathcal{R}-\int_{0}^{1}c_{ex}%
\left( x\right) dx=\Delta C\left( \left[ 0,1\right] \right) -0=1\text{.}
\end{equation*}

So, the sum of the changes in the value of $C$ over $\mathcal{C}$ is reduced
to the so-called indeterminate expression $\infty \cdot 0$ (the residue
function $\mathfrak{R}$ of $C$ vanishes identically on $\left[ 0,1\right] $
because $C$ is continuous on $\left[ 0,1\right] $), that actually have, in
this situation, the real numerical value of $1$ (it means that $C$ is not
absolutely continuous and has no negligible variation on $\mathcal{C}$).
Let's prove it once more. For the\textit{\ Cantor} function with the total
length of $2$ on $\left[ 0,1\right] $ the total length of all line segments
contained within $\left[ 0,1\right] \backslash \mathcal{C}$, on each of
which $C$ is constant, is as follows 
\begin{equation*}
\frac{1}{2}\sum_{n=1}^{+\infty }(\frac{2}{3})^{n}=\frac{1}{2}(3-1)=1\text{.}
\end{equation*}%
Hence, the sum of the changes in the value of $C$ over $\mathcal{C}$, is
equal to $2-1$.

\end{document}